\documentclass{article}

\usepackage{authblk}

\usepackage{amsmath}
\usepackage{amsthm}
\usepackage{amsfonts}
\usepackage{amssymb}
\usepackage[all]{xy}
\usepackage{url}

\DeclareMathSymbol{\rightrightarrows}  {\mathrel}{AMSa}{"13}

\def\sd{\operatorname{sd}}

\def\sk{\operatorname{sk}}

\def\Re{\operatorname{Re}}

\catcode`\@=11
\def\varholim@#1#2{\mathop{\vtop{\ialign{##\crcr
 \hfil$#1\m@th\operator@font holim$\hfil\crcr
 \noalign{\nointerlineskip\kern\ex@}#2#1\crcr
 \noalign{\nointerlineskip\kern-\ex@}\crcr}}}}
\def\hocolim{\mathpalette\varholim@\rightarrowfill@} 
\def\hoinvlim{\mathpalette\varholim@\leftarrowfill@}
\catcode`\@=\active

\newtheorem{theorem}{Theorem}
\newtheorem{lemma}[theorem]{Lemma}

\newtheorem{proposition}[theorem]{Proposition}

\theoremstyle{definition}

\newtheorem{example}[theorem]{Example}

\newtheorem{remark}[theorem]{Remark}

\begin{document}

\title{\bf Metric spaces and homotopy types}
\author{J.F. Jardine\thanks{Supported by NSERC.}}

\affil{\small Department of Mathematics\\University of Western Ontario\\
  London, Ontario, Canada
}
\affil{jardine@uwo.ca}

\maketitle

\begin{abstract}
By analogy with methods of Spivak, there is a realization functor which takes a persistence diagram $Y$ in simplicial sets to an extended pseudo-metric space (or ep-metric space) $\Re(Y)$. The functor $\Re$ has a right adjoint, called the singular functor, which takes an ep-metric space $Z$ to a persistence diagram $S(Z)$. We give an explicit description of $\Re(Y)$, and show that it depends only on the $1$-skeleton $\sk_{1}Y$ of $Y$. If $X$ is a totally ordered ep-metric space, then there is an isomorphism $\Re(V_{\ast}(X)) \cong X$, between the realization of the Vietoris-Rips diagram $V_{\ast}(X)$ and the ep-metric space $X$. The persistence diagrams $V_{\ast}(X)$ and $S(X)$ are sectionwise equivalent for all such $X$.
  \end{abstract}

\section*{Introduction}

An extended pseudo-metric space, here called an ep-metric space, is a set $X$ together with a function $d:X \times X \to [0,\infty]$ such that the following conditions hold:
  \begin{itemize}
  \item[1)] $d(x,x)=0$,
  \item[2)] $d(x,y) = d(y,x)$,
  \item[3)] $d(x,z) \leq d(x,y) + d(y,z)$.
  \end{itemize}
  There is no condition that $d(x,y)=0$ implies $x$ and $y$ coincide --- this is where the adjective ``pseudo'' comes from, and the gadget is ``extended'' because we are allowing an infinite distance.

  A metric space is an ep-metric space for which $d(x,y)=0$ implies $x=y$, and all distances $d(x,y)$ are finite.

  The traditional objects of study in topological data analysis are finite metric spaces $X$, and the most common analysis starts by creating a family of simplicial complexes $V_{s}(X)$, the Vietoris-Rips complexes for $X$, which are parameterized by a distance variable $s$.

  To construct the complex $V_{s}(X)$, it is harmless at the outset is to list the elements of $X$, or give $X$ a total ordering --- one can always do this without damaging the homotopy type. Then $V_{s}(X)$ is a simplicial complex (and a simplicial set), with simplices given by strings
  \begin{equation*}
    x_{0} \leq x_{1} \leq \dots \leq x_{n}
    \end{equation*}
of elements of $X$ such that $d(x_{i},x_{j}) \leq s$ for all $i,j$. If $s \leq t$ then there is an inclusion $V_{s}X \subset V_{t}(X)$, and varying the distance parameter $s$ gives a diagram (functor) $V_{\ast}(X): [0,\infty] \to s\mathbf{Set}$, taking values in simplicial sets.

Following Spivak \cite{fuzzy-Spivak} (sort of), one can take an arbitrary diagram $Y: [0,\infty] \to s\mathbf{Set}$, and produce an ep-metric space $\Re(Y)$, called its realization. This realization functor has a right adjoint $S$, called the singular functor, which takes an ep-metric space $Z$ and produces a diagram $S(Z): [0,\infty] \to s\mathbf{Set}$ in simplicial sets.

One needs good cocompleteness properties to construct the realization functor $\Re$. Ordinary metric spaces are not well behaved in this regard, but it is shown in the first section (Lemma \ref{lem 3}) that the category of ep-metric spaces has all of the colimits one could want. Then $\Re(Y)$ can be constructed as a colimit of finite metric spaces $U^{n}_{s}$, one for each simplex $\Delta^{n} \to Y_{s}$ of some section of $Y$.

The metric space $U^{n}_{s}$ is the set $\{0,1, \dots ,n\}$, equipped with a metric $d$, where $d(i,j) = s$ for $i \ne j$. A morphism $U^{n}_{s} \to Z$ of ep-metric spaces is a list $(x_{0},x_{1}, \dots ,x_{n})$ of elements of $Z$ such that $d(x_{i},x_{j}) \leq s$ for all $i,j$. Such lists have nothing to with orderings on $Z$, and could have repeats.

With a bit of categorical homotopy theory, one shows (Proposition \ref{prop 7}) that $\Re(Y)$ is the set of vertices of the simplicial set $Y_{\infty}$ (evaluation of $Y$ at $\infty$), equipped with a metric that is imposed by the proof of Lemma \ref{lem 3}.

One wants to know about the homotopy properties of the counit map $\eta: Y \to S(\Re(Y))$, especially when $Y$ is an old friend such as the Vietoris-Rips system $V_{\ast}(X)$. But $\Re(V_{\ast}(X))$ is the original metric space $X$ (Example \ref{ex 13}), the object $S(X)$ is the diagram $[0,\infty] \to s\mathbf{Set}$ with $(S(X)_{t})_{n} = \hom(U^{n}_{t},X)$, and the counit $\eta: V_{t}(X) \to S_{t}(X)$ in simplicial sets takes an $n$-simplex $\sigma: \Delta^{n} \to V_{t}(X)$ to the list $(\sigma(0),\sigma(1), \dots ,\sigma(n))$ of its vertices.

We show in Section 3 (Theorem \ref{th 16}, the main result of this paper) that the map $\eta: V_{t}(X) \to S_{t}(X)$ is a weak equivalence for all distance parameter values $t$. The proof proceeds in two main steps, and involves technical results from the theory of simplicial approximation. The steps are the following:
\smallskip

\noindent
1)\ We show (Lemma \ref{lem 14}) that the map $\eta$ induces a weak equivalence $\eta_{\ast}: BNV_{t}(X) \to BNS_{t}(X)$, where $\eta_{\ast}: NV_{t}(X) \to NS_{t}(X)$ is the induced comparison of posets of non-degenerate simplices. Here, $V_{t}(X)$ is a simplicial complex, so that $BNV_{t}(X)$ is a copy of the subdivision $\sd(V_{t}(X))$, and is therefore weakly equivalent to $V_{t}(X)$.
\medskip

\noindent
2)\ There is a canonical map $\pi: \sd S_{t}(X) \to BNS_{t}(X)$, and the second step in the proof of Theorem \ref{th 16} is to show (Lemma \ref{lem 15}) that this map $\pi$ is a weak equivalence.
\smallskip

\noindent
It follows that the map $\eta$ induces a weak equivalence $\sd(V_{t}(X)) \to \sd(S_{t}(X))$, and Theorem \ref{th 16} is a consequence.
\medskip

The fact that the space $S_{t}(X)$ is weakly equivalent to $V_{t}(X)$ for each $t$ means that we have yet another system of spaces $S_{\ast}(X)$ that models persistent homotopy invariants for a data set $X$.

One should bear in mind, however, that $S_{t}(X)$ is an infinite complex. To see this, observe that if $x_{0}$ and $x_{1}$ are distinct points in $X$ with $d(x_{0},x_{1}) \leq t$, then all of the lists
  \begin{equation*}
    (x_{0},x_{1},x_{0},x_{1}, \dots ,x_{0},x_{1})
  \end{equation*}
  define non-degenerate simplices of $S_{t}(X)$.

\tableofcontents

\section{ep-metric spaces}

  An {\it extended pseudo-metric space} \cite{HMc-2018} (or an {\it uber metric space} \cite{fuzzy-Spivak}) is a set $Y$, together with a function
  $d:Y \times Y \to [0,\infty]$, such that the following conditions hold:
  \begin{itemize}
  \item[a)] $d(x,x)=0$,
  \item[b)] $d(x,y) = d(y,x)$,
  \item[c)] $d(x,z) \leq d(x,y) + d(y,z)$.
  \end{itemize}
Following \cite{Scocc-thesis}, I use the term {\it ep-metric spaces} for these objects, which will be denoted by $(Y,d)$ in cases where clarity is required for the metric.
\medskip

  Every metric space $(X,d)$ is an ep-metric space, by composing the distance function $d: X \times X \to [0,\infty)$ with the inclusion $[0,\infty) \subset [0,\infty]$.

\medskip      

  A morphism between ep-metric spaces $(X,d_{X})$ and $(Y,d_{Y})$ is a function $f: X \to Y$ such that
  \begin{equation*}
    d_{Y}(f(x),f(y)) \leq d_{X}(x,y).
  \end{equation*}
  These morphisms are sometimes said to be non-expanding \cite{HMc-2018}.

  I shall use the notation $ep-\mathbf{Met}$ to denote the category of ep-metric spaces and their morphisms.
  \medskip

 \begin{example}[Quotient ep-metric spaces]\label{ex 1}
  Suppose that $(X,d)$ is an ep-metric space and that $p: X \to Y$ is a surjective function.

For $x,y \in Y$, set
\begin{equation}\label{eq 1}
  D(x,y) = \inf_{P}\ \sum_{i}\ d(x_{i},y_{i}),
\end{equation}
where each $P$ consists of pairs of points $(x_{i},y_{i})$ with $x=x_{0}$ and $y_{k}=y$, such that $p(y_{i})=p(x_{i+1})$.

Certainly $D(x,x) = 0$ and $D(x,y) = D(y,x)$. One thinks of each $P$ in the definition of $D(x,y)$ as a ``polygonal path'' from $x$ to $y$. Polygonal paths concatenate, so that $D(x,z) \leq D(x,y) + D(y,z)$, and $D$ gives the set $Y$ an ep-metric space structure. This is the {\it quotient} ep-metric space structure on $Y$.

If $x,y$ are elements of $X$, the pair $(x,y)$ is a polygonal path from $x$ to $y$, so that $D(p(x),p(y)) \leq d(x,y)$. It follows that the function $p$ defines a morphism $p: (X,d) \to (Y,D)$ of ep-metric spaces.
\end{example}

 \begin{example}[Dividing by zero]\label{ex 2}
   Suppose that $(X,d)$ is an ep-metric space. There is an equivalence relation on $X$, with $x \sim y$ if and only if $d(x,y) =0$. Write $p: X \to X/\sim\ =: Y$ for the corresponding quotient map.

   Given a polygonal path $P = \{ (x_{i},y_{i}) \}$ from $x$ to $y$ in $X$ as above, $d(y_{i},x_{i+1}) = 0$, so the sum corresponding to $P$ in (\ref{eq 1}) can be rewritten as
   \begin{equation*}
     d(x,y_{0}) + d(y_{0},x_{1}) + d(x_{1},y_{1}) + \dots + d(x_{k},y).
   \end{equation*}
It follows that $d(x,y) \leq D(p(x),p(y))$, whereas $D(p(x),p(y)) \leq d(x,y)$ by construction.

Thus, if $D(p(x),p(y)) = 0$, then $d(x,y)=0$ so that $p(x)=p(y)$.
 \end{example}

  \begin{lemma}\label{lem 3}
The category $ep-\mathbf{Met}$ of ep-metric spaces is cocomplete.
  \end{lemma}

  \begin{proof}
    The empty set is the initial object for this category,

    Suppose that $(X_{i},d_{i}), i \in I$, is a list of ep-metric spaces. Form the set theoretic disjoint union $X = \sqcup_{i}\ X_{i}$, and define a function
    \begin{equation*}
      d: X \times X \to [0,\infty]
    \end{equation*}
    by setting $d(x,y) = d_{i}(x,y)$ if $x,y$ belong to the same summand $X_{i}$ and $d(x,y)= \infty$ otherwise. Any collection of morphisms $f_{i}: X_{i} \to Y$ in $ep-\mathbf{Met}$ defines a unique function $f=(f_{i}): X \to A$, and this function is a morphism of $ep-\mathbf{Met}$ since
    \begin{equation*}
      d(f(x),f(y)) = d(f_{i}(x),f_{j}(y)) \leq \infty = d(x,y)
    \end{equation*}
    if $x \in X_{i}$ and $y \in X_{j}$ with $i \ne j$.

    Suppose given a pair of morphisms
    \begin{equation*}
      \xymatrix{
        A \ar@<1ex>[r]^{f} \ar@<-1ex>[r]_{g} & X
      }
    \end{equation*}
    in $ep-\mathbf{Met}$, and form the set theoretic coequalizer $\pi: X \to C$. The function $p$ is the canonical map onto a set of equivalence classes of $X$, which classes are defined by the relations $f(a) \sim g(a)$ for $a \in A$. We give $C$ the quotient ep-metric space structure, as in Example \ref{ex 1}.

Suppose that $\alpha: (X,d_{X}) \to (Z,d_{Z})$ is an morphism of ep-metric spaces such that $\alpha \cdot f = \alpha \cdot g$. Write $\alpha_{\ast}: C \to Z$ for the unique function such that $\alpha_{\ast} \cdot p = \alpha$.
    
    Suppose given a polygonal path $P = \{ (x_{i},y_{i}) \}$ from $x$ to $y$ in $X$. Then $\alpha(y_{i}) = \alpha(x_{i+1})$, so that
    \begin{equation*}
      d_{Y}(\alpha(x),\alpha(y)) \leq \sum_{i}\ d_{Y}(\alpha(x_{i}),\alpha(y_{i}))
      \leq \sum_{i}\ d_{X}(x_{i},y_{i}).
    \end{equation*}
    This is true for every polygonal path from $x$ to $y$ in $X$, so that
    \begin{equation*}
      d_{Y}(\alpha_{\ast}p(x),\alpha_{\ast}p(y)) \leq d_{C}(p(x),p(y)).
      \end{equation*}
It follows that $\alpha_{\ast}: (C,d_{C}) \to (Z,d_{Z})$ is a morphism of ep-metric spaces.    
  \end{proof}

\begin{example}[``Bad'' filtered colimit]\label{ex 4}
  If one starts with a diagram of metric spaces, the colimit $C$ that is produced by Lemma \ref{lem 3} is an ep-metric space,  and it may be that $d(x,y) = 0$ 
in the coequalizer $C$ for some elements $x,y$ with $x \ne y$.

    In particular, suppose that $X_{s} = \{ (\frac{1}{s\sqrt 2},0),(0,\frac{1}{s\sqrt 2})\} \subset \mathbb{R}^{2}$ for $0 < s < \infty$.
    Write $p_{s} = (\frac{1}{s\sqrt 2},0)$ and $q_{s} = (0,\frac{1}{s\sqrt 2})$ in $X_{s}$. Then $d(p_{s},q_{s}) = \frac{1}{s}$.
    For $s \leq t$ there is an ep-metric space map $X_{s} \to X_{t}$ which is defined by $p_{s} \mapsto p_{t}$ and $q_{s} \mapsto q_{t}$. 

    The filtered colimit $\varinjlim_{s}\ X_{s}$ has two distinct points, namely $p_{\infty}$ and $q_{\infty}$, and $d(p_{\infty},q_{\infty}) \leq d(p_{s},q_{s}) = \frac{1}{s}$ for all $s >0$. It follows that $d(p_{\infty},q_{\infty}) = 0$, whereas $p_{\infty} \ne q_{\infty}$.
\end{example}

\begin{lemma}\label{lem 5}
  Suppose that $X$ is an ep-metric space. Then there is an isomorphism of ep-metric spaces
  \begin{equation*}
  \psi:  \varinjlim_{F}\ F \xrightarrow{\cong} X,
    \end{equation*}
where $F$ varies over the finite subsets of $X$, with their induced ep-metric space structures.
\end{lemma}

\begin{proof}
  The collection of finite subsets of $X$ is filtered, and the set $X$ is a filtered colimit of its finite subsets, so the function defining the ep-metric space map $\psi$ is a bijection. Write $d_{\infty}$ for the metric on the filtered colimit.

  If $x,y \in X$ and $d(x,y) = s \leq \infty$ in $X$, then there is a finite subset $F$ with $x,y \in F$ such that $d(x,y) = s$ in $F$. The list $(x,y)$ is a polygonal path from $x$ to $y$ in $F$, so that $d_{\infty}(x,y) \leq d(x,y)$. It follows that $d(x,y) = d_{\infty}(x,y)$, and so $\psi$ is an isomorphism.  
  \end{proof}

An ep-metric space $(X,d)$ has an associated system of posets $P_{\ast}(X): [0,\infty] \to s\mathbf{Set}$, where $P_{s}(X)$ is the collection of finite subsets $F$ of $X$ such that $d(x,y) \leq s$ for any two members $x,y$ of $X$.

This construction defines a system of abstract simplicial complexes $V_{\ast}(X)$, which can be constructed entirely within simplicial sets when $X$ has a total ordering. In that case, the $n$-simplices of the simplicial set $V_{s}(X)$ are the strings $x_{0} \leq x_{1} \leq \dots \leq x_{n}$ such that $d(x_{i},x_{j}) \leq s$. The diagram $V_{\ast}(X): [0,\infty] \to s\mathbf{Set}$ is the Vietoris-Rips system. The spaces $V_{s}(X)$ are independent up to weak equivalence of the ordering on $X$, because there is a canonical weak equivalence (a ``last vertex map'') $\gamma: BP_{s}(X) \to V_{s}(X)$ of systems, while the spaces $BP_{s}(X)$ are defined independently from the ordering. In classical terms, the nerve $BP_{s}(X)$ of the poset $P_{s}(X)$ (non-degenerate simplices of the Vietoris-Rips complex $V_{s}(X)$) is the barycentric subdivision of $V_{s}(X)$.

\begin{example}[Excision for path components]\label{ex 6}
Suppose that $X$ and $Y$ are finite subsets of an ep-metric space $Z$, with the induced ep-metric space structures. Consider the inclusions of finite ep-metric spaces
      \begin{equation*}
        \xymatrix{
          X \cap Y \ar[r] \ar[d] & Y \ar[d] \\
          X \ar[r] & X \cup Y
        }
      \end{equation*}
      inside $Z$. Write $X \cup_{m} Y$ for the corresponding pushout in the category of ep-metric spaces. The unique map
\begin{equation*}
  X \cup_{m} Y \to X \cup Y
  \end{equation*}
of ep-metric spaces is the identity on the underlying point set. Write $d_{m}$ for the metric on $X \cup_{m} Y$. Then $d_{m}(x,y)$ is the minimum of sums
\begin{equation}\label{eq 2}
  \sum\ d(x_{i},x_{i+1}),
  \end{equation}
  indexed over paths
  \begin{equation*}
    P:\ x = x_{0},x_{1}, \dots .x_{n}=y,
  \end{equation*}
  such that for each $i$ the points $x_{i},x_{i+1}$ are either both in $X$ or both in $Y$.

  All sums in (\ref{eq 2}) are finite, and $d_{m}(x,y)$ is realized by a particular path $P$ since $X$ and $Y$ are finite. Note that $d(x,y) \leq d_{m}(x,y)$, by construction, and that $d(x,y) = d_{m}(x,y)$ if $x,y$ are both in either $X$ or $Y$.

  There are induced simplicial set maps
  \begin{equation*}
    V_{s}(X) \cup V_{s}(Y) \to V_{s}(X \cup_{m} Y) \to V_{s}(X \cup Y),
  \end{equation*}
  all of which are the identity on vertices. There is a $1$-simplex $\sigma = \{x,y\}$ of $V_{s}(X \cup_{m} Y)$ if and only if there is a path
    \begin{equation*}
     P:\ x=x_{0},x_{1}, \dots ,x_{n}=y
    \end{equation*}
    consisting of $1$-simplices in either $X$ or $Y$, such that
    \begin{equation*}
      \sum d(x_{i},x_{i+1}) \leq s.
      \end{equation*}
    Then all $d(x_{i},x_{i+1}) \leq s$, so that $x$ and $y$ are in the same path component of $V_{s}(X) \cup V_{s}(Y)$. It follows that there is an induced isomorphism
    \begin{equation}\label{eq 3}
      \pi_{0}(V_{s}(X) \cup V_{s}(Y)) \cong \pi_{0}V_{s}(X \cup_{m} Y).
    \end{equation}

    The isomorphisms (\ref{eq 3}) induce isomorphisms
\begin{equation}\label{eq 4}
      \pi_{0}(V_{s}(X) \cup V_{s}(Y)) \cong \pi_{0}V_{s}(X \cup_{m} Y).
    \end{equation}
 for arbitrary subsets $X$ and $Y$ of an ep-metric space $Z$, by an application of Lemma \ref{lem 5}. 
\end{example}
    
    \section{Metric space realizations}

    Write $U^{n}_{s}$ for the collection of axis points $x_{i}=\frac{s}{\sqrt 2}e_{i}$, where
\begin{equation*}
  e_{i} = (0, \dots ,\overset{i+1}{1}, \dots ,0) \in  \mathbb{R}^{n+1}.
\end{equation*}
for $0 \leq i \leq n$.
Observe that $d(x_{i},x_{j}) = s$ in $\mathbb{R}^{n+1}$ for $i \ne j$. Another way of looking at it: $U^{n}_{s}$ is the set $\mathbf{n} = \{0,1,\dots ,n\}$ with $d(i,j) = s$ for $i \ne j$.

An ep-metric space morphism $f: U^{n}_{s} \to Y$ consists of points $f(x_{i})$, $0 \leq i \leq n$, such that $d_{Y}(f(x_{i}),f(x_{j})) \leq s$ for all $i,j$.
  \medskip

  Write $s\mathbf{Set}^{[0,\infty]}$ for the category of diagrams (functors) $X: [0,\infty] \to s\mathbf{Set}$ and their natural transformation, which take values in simplicial sets and are
  defined on the poset $[0,\infty]$. I usually write $s \mapsto X_{s}$ for such a diagram $X$. In particular, $X_{\infty}$ is the value that the diagram $X$ takes at the terminal object of $[0,\infty]$.

 Suppose that $K$ is a simplicial set. The representable diagram $L_{s}K$ satisfies the universal property
  \begin{equation*}
    \hom(L_{s}K,X) \cong \hom(K,X_{s}).
  \end{equation*}
  One  shows that
  \begin{equation*}
    (L_{s}K)_{t} =
    \begin{cases}
      \emptyset & \text{if $t<s$,} \\
      K & \text{if $t \geq s$.}
    \end{cases}
  \end{equation*}
  The set of maps $L_{s}\Delta^{n} \to X$ can be identified with the set of $n$-simplices of the simplicial set $X_{s}$.

A morphism $L_{t}\Delta^{m} \to L_{s}\Delta^{n}$ consists of a relation $s \leq t$ and a simplicial map $\theta: \Delta^{m} \to \Delta^{n}$.
In the presence of such a morphism,
the function $\theta: \mathbf{m} \to \mathbf{n}$ defines an
ep-metric space morphism $U^{m}_{t} \to U^{n}_{s}$, since $s=d(\theta(i),\theta(j)) \leq d(i,j)=t$.

\subsection{The realization functor}

Suppose that $X: [0,\infty] \to s\mathbf{Set}$ is a diagram. The category $\mathbf{\Delta}/X$ of simplices of $X$ has maps $L_{s}\Delta^{n} \to X$ as objects and commutative diagrams
\begin{equation*}
  \xymatrix@C=10pt{
    L_{t}\Delta^{m} \ar[rr]^{\theta} \ar[dr]_{\tau} && L_{s}\Delta^{n} \ar[dl]^{\sigma} \\
    & X
  }
\end{equation*}
as morphisms.

Equivalently, a simplex of $X$ is a simplicial set map $\Delta^{n} \to X_{s}$, and a morphism of simplices is a diagram
\begin{equation}\label{eq 5}
  \xymatrix{
    \Delta^{m} \ar[r]^{\theta} \ar[d]_{\tau} & \Delta^{n} \ar[d]^{\sigma} \\
    X_{t} & X_{s} \ar[l]
  }
\end{equation}

Every simplex $\Delta^{n} \to X_{s}$  determines a simplex
\begin{equation*} 
  \Delta^{n} \to X_{s} \to X_{\infty},
\end{equation*}
and we have a functor $r: \mathbf{\Delta}/X \to \mathbf{\Delta}/X_{\infty}$, where $\mathbf{\Delta}/X_{\infty}$ is the simplex category of the simplicial set $X_{\infty}$.

There is an inclusion $i: \mathbf{\Delta}/X_{\infty} \to \mathbf{\Delta}/X$, and the composite $r \cdot i$ is the identity. The maps
\begin{equation}\label{eq 6}
  \xymatrix{
    \Delta^{n} \ar[r]^{1} \ar[d]_{r(\sigma)} & \Delta^{n} \ar[d]^{\sigma} \\
    X_{\infty} & X_{s} \ar[l]
  }
  \end{equation}
define a natural transformation $h: i\cdot r \to 1$.

There is a functor
\begin{equation}\label{eq 7}
  \mathbf{\Delta}/X \to s\mathbf{Set}
\end{equation}
  which takes a morphism (\ref{eq 5}) to the map $\theta: \Delta^{m} \to \Delta^{n}$.

The translation category $E_{X}$ for the functor (\ref{eq 7}) is a simplicial category that has objects consisting of pairs $(\sigma,x)$ where $\sigma: \Delta^{n} \to X_{s}$ and $x \in \Delta^{n}$ (of a fixed dimension). A morphism $(\tau,y) \to (\sigma,x)$ of $E_{X}$ is a morphism $\theta: \tau \to \sigma$ as in (\ref{eq 5}) such that $\theta(y) = x$. The path component simplicial set $\pi_{0}E_{X}$ of the category $E_{X}$ is isomorphic to the colimit
\begin{equation*}
  \varinjlim_{L_{s}\Delta^{n} \to X}\ \Delta^{n}.
\end{equation*}

There is a correponding translation category $E_{X_{\infty}}$ for the functor which takes the simplex $\Delta^{n} \to X_{\infty}$ to the simplicial set $\Delta^{n}$, and there is an induced functor $i_{\ast}: E_{X_{\infty}} \subset E_{X}$. The functor $r: \mathbf{\Delta}/X \to \mathbf{\Delta}/X_{\infty}$ induces a functor $r_{\ast}: E_{X} \to E_{X_{\infty}}$. The composite $r_{\ast} \cdot i_{\ast}$ is the identity on $E_{X_{\infty}}$, and
the map (\ref{eq 6}) defines a natural transformation $i_{\ast} \cdot r_{\ast} \to 1$ of functors $E_{X_{\infty}} \to E_{X_{\infty}}$.

The translation categories $E_{X}$ and $E_{X_{\infty}}$ are therefore homotopy equivalent, and thus have isomorphic simplicial sets of path components. It follows that there are isomorphisms of simplicial sets
\begin{equation}\label{eq 8}
X_{\infty} \xleftarrow{\cong} \varinjlim_{\Delta^{n} \to X_{\infty}}\ \Delta^{n} \xrightarrow{\cong}  \varinjlim_{L_{s}\Delta^{n} \to X}\ \Delta^{n}
\end{equation}


Suppose that $X: [0, \infty] \to s\mathbf{Set}$ is a diagram,
and set
\begin{equation*}
  \Re(X) = \varinjlim_{L_{s}\Delta^{n} \to X}\ U^{n}_{s}
\end{equation*}
in the category of ep-metric spaces.

It follows from the identifications of (\ref{eq 8}) that $\Re(X)$ is the set of vertices of $X_{\infty}$, equipped with an ep-metric space structure.

If $x$ and $y$ are two such vertices, and are the boundary of a $1$-simplex
\begin{equation*}
  \Delta^{1} \to X_{s} \to X_{\infty}
\end{equation*}
  then $x$ and $y$ are in the image of a map $U^{1}_{s} \to \Re(X)$, so that $d(x,y) \leq s$. If there is a sequence of $1$-simplices $\omega_{i}: \Delta^{1} \to X_{s_{i}}$ that define a polygonal path
\begin{equation*}
P:  x=x_{0} \leftrightarrows x_{1} \leftrightarrows \dots \leftrightarrows x_{k} =y
\end{equation*}
of $1$-simplices in $X_{\infty}$, then $d(x,y) \leq \sum_{i}\ s_{i}$ by definition. Formally,  we set
\begin{equation}\label{eq 9}
  d(x,y) = \inf_{P}\ \{ \sum_{i}\ s_{i}\}.
  \end{equation}
provided such polygonal paths exist. Otherwise, we set $d(x,y) = \infty$.

The resulting metric $d$ is the metric which is imposed on the set of vertices of $X_{\infty}$ by the requirement that
\begin{equation*}
  \Re(X) = \varinjlim_{L_{s}\Delta^{n} \to X}\ U^{n}_{s}
\end{equation*}
in the category of ep-metric spaces --- see Lemma \ref{lem 3}. We have shown the following:

\begin{proposition}\label{prop 7}
  Suppose that $X: [0,\infty] \to s\mathbf{Set}$ is a diagram. Then the ep-metric space $\Re(X)$ has underlying set given by the set of vertices of $X_{\infty}$, with metric defined within path components by (\ref{eq 9}). Elements $x$ and $y$ that are in distinct path components have $d(x,y) = \infty$.  
  \end{proposition}

\begin{example}[Realization of Vietoris-Rips systems]\label{ex 8}
  Suppose that $X$ is a finite ep-metric space, and that $X$ is totally ordered.
  
    The realization $\Re(V_{\ast}(X))$ has $X$ as its underlying set, and $V_{\infty}(X) = \Delta^{X}$ is a finite simplex, which is connected, so that there is a finite polygonal path in $X$ between any two points $x,y \in X$. We have a relation
    \begin{equation*}
      d(x,y) \leq \sum_{i}\ s_{i}
    \end{equation*}
   in $X$ for any polygonal path $P$ which is defined by $1$-simplices $\omega_{i} \in V_{s_{i}}(X)$. This means that $d(x,y)$ in $X$ coincides with the distance between $x$ and $y$ in the metric space $\Re(X)$. It follows that the identity on the set $X$ induces an isomorphism of ep-metric spaces
\begin{equation*}
  \phi: \Re(V_{\ast}(X)) \xrightarrow{\cong} X.
\end{equation*}

This map $\phi$ is an isomorphism of ep-metric spaces, by Lemma \ref{lem 5} and the previous paragraphs.
\end{example}

\begin{example}[Degree Rips systems]\label{ex 9}
  Continue with a finite totally ordered ep-metric space $X$ as in Example \ref{ex 8}, let $k$ be a positive integer, and consider the degree Rips system $L_{\ast,k}(X)$. We choose $k$ such that the system of complexes $L_{\ast,k}(X)$ is non-empty, ie. such that $k \leq \vert X \vert$. Then $L_{t,k}(X) = V_{t}(X)$ for $t$ sufficiently large, and $X$ is the underlying set of $\Re(L_{\ast,k}(X))$.

  The maps
  \begin{equation*}
      \Re(L_{\ast,k}(X)) \to \Re(V_{\ast}(X)) \to X
  \end{equation*}
  are isomorphisms of metric spaces, by a cofinality argument.
\end{example}

Here is a special case:

\begin{lemma}
  Suppose that $K$ is a simplicial complex and that $s > 0$. Then $\Re(L_{s}K) = \Re(L_{s}\sk_{1}(K))$, and $\Re(L_{s}K)$ is the set of vertices $K_{0}$ with a metric $d$ defined by
  \begin{equation*}
    d(x,y) =
    \begin{cases}
      \infty & \text{if $[x] \ne [y]$ in $\pi_{0}(K)$,} \\
      \min_{P}\ s \cdot k & \text{if $[x]=[y]$.}
    \end{cases}
  \end{equation*}
  where $P$ varies through the polygonal paths
  \begin{equation*}
    P: x = x_{0} \leftrightarrows x_{1} \leftrightarrows \dots \leftrightarrows x_{k} =y
  \end{equation*}
  of $1$-simplices between $x$ and $y$.
\end{lemma}

\begin{proof}
  Write $\Re(K) = \Re(L_{s}K)$. The simplicial set $K$ is a colimit of its simplices, and so there is an isomorphism
  \begin{equation*}
    \varinjlim_{\Delta^{n} \to K}\ L_{s}\Delta^{n} \xrightarrow{\cong} L_{s}K.
  \end{equation*}
  It follows that there is an isomorphism
  \begin{equation*}
    \varinjlim_{\Delta^{n} \to K}\ U^{n}_{s} \xrightarrow{\cong} \Re(L_{s}K).
  \end{equation*}

Suppose that $n \geq 2$. Then $\partial\Delta^{n}$ and $\Delta^{n}$ have the same vertices, and any two vertices $x,y$ are on a common face $\Delta^{n-1} \subset \Delta^{n}$. It follows that $d(x,y)=s$ in $\Re(\partial\Delta^{n})$ and $\Re(\Delta^{n})$, and the induced map
\begin{equation*}
  \Re(\partial\Delta^{n}) \to \Re(\Delta^{n})
\end{equation*}
is an isomorphism for $n \geq 2$.

The displayed metric $d$ on the vertices of $K$ defines a metric space $\Re(K)$, with maps $\sigma_{\ast}: U^{n}_{s} \to \Re(K)$ for all simplices $\sigma: \Delta^{n} \to K$, which maps are natural with respect to the simplicial structure of $K$.

Any family of metric space morphisms $f_{\sigma}: U^{n}_{s} \to Y$ determines a unique function $f: K_{0} \to Y$. Also, $d(f(x),f(y)) \leq s$ if $x,y$ are in a common simplex $\Delta^{1} \to K$. If $P$ is a polygonal path between $x$ and $y$ as above, then $d(f(x),f(y)) \leq k \cdot s$. This is true for all such polygonal paths, so $d(f(x),f(y)) \leq d(x,y)$.

If $x$ and $y$ are in distinct components of $K$, then $d(f(x),f(y)) \leq d(x,y) = \infty$.
\end{proof}

The following result says that the realization $\Re(X)$ of a diagram $X: [0,\infty] \to s\mathbf{Set}$ depends only on the associated diagram of graphs $\sk_{1}(X)$.

\begin{lemma}
  Suppose that $X: [0,\infty] \to s\mathbf{Set}$ is a diagram. Then the inclusion $\sk_{1}X \subset X$ induces an isomorphism
\begin{equation*}
\Re(\sk_{1}X) \xrightarrow{\cong} \Re(X).
\end{equation*}
\end{lemma}

\begin{proof}
  The diagram of $1$-skeleta $\sk_{1}X$ is a colimit
\begin{equation*}
  \varinjlim_{L_{s}\Delta^{n} \to X}\ L_{s}\sk_{1}\Delta^{n},
\end{equation*}
since the functor $\sk_{1}$ preserves colimits. There are commutative diagrams
\begin{equation*}
  \xymatrix{
    \Re(L_{s}\sk_{1}\Delta^{n}) \ar[r] \ar[d]_{\cong} & \Re(\sk_{1}X) \ar[d] \\
    \Re(L_{s}\Delta^{n}) \ar[r] & \Re(X)
  }
  \end{equation*}
that are natural in the simplices of $X$, and it follows that the induced map $\Re(\sk_{1}X) \to \Re(X)$ is an isomorphism, as required.
\end{proof}

\subsection{Partial realizations}

Suppose again that $X: [0,\infty] \to s\mathbf{Set}$ is a diagram in simplicial sets. We construct partial realizations by writing
\begin{equation*}
  \Re(X)_{s} = \varinjlim_{L_{t}\Delta^{n} \to X,\ t\leq s}\ U^{n}_{t}.
\end{equation*}
This is the colimit of a functor taking values in ep-metric spaces, which is defined on the full subcategory $\mathbf{\Delta}/X_{\leq s}$ of $\mathbf{\Delta}/X$ having objects $L_{t}\Delta^{n} \to X$ with $t \leq s$. A map $L_{t}\Delta^{n} \to X$ can be identified with a simplex $\Delta^{n} \to X_{t}$, and the relation $t \leq s$ defines a simplex $\Delta^{n} \to X_{t} \to X_{s}$, so that we have a functor
\begin{equation*}
  r: \Delta/X_{\leq s} \to \mathbf{\Delta}/X_{s},
\end{equation*}
along with an inclusion $i: \mathbf{\Delta}/X_{s} \subset \mathbf{\Delta}/X_{\leq s}$. The composite $r \cdot i$ is the identity, and the composite $i \cdot r$ is homotopic to the identity, just as before.

There is a functor $\mathbf{\Delta}/X_{\leq s} \to s\mathbf{Set}$ which takes a simplex $\Delta^{n} \to X_{t}$ to the simplicial set $\Delta^{n}$. By manipulating path components of homotopy colimits, one finds isomorphisms
  \begin{equation*}
    X_{s} \xleftarrow{\cong} \varinjlim_{\Delta^{n} \to X_{s}}\ \Delta^{n}
    \xrightarrow{\cong} \varinjlim_{L_{t}\Delta^{n} \to X,\ t\leq s}\ \Delta^{n}.
    \end{equation*}
 that are analogous to the isomorphisms of (\ref{eq 9}). It follows, as in Theorem \ref{prop 7}, that the set underlying the metric space $\Re(X)_{s}$ is the set of vertices of the simplicial set $X_{s}$.   

 The metric $d$ on $(X_{s})_{0}$ is defined as before: $d(x,y) = \infty$ if $x$ and $y$ not in the same path component of $X_{s}$.  Otherwise
 \begin{equation}\label{eq 10}
   d(x,y) = \inf_{P}\ \{\sum t_{i}\},
 \end{equation}
 indexed over all polygonal paths
 \begin{equation*}
   P: x = x_{0} \leftrightarrows x_{1} \leftrightarrows \dots \leftrightarrows x_{k} =y
 \end{equation*}
 that are defined by $1$-simplices $\omega: \Delta^{1} \to X_{t_{i}}$ with $t_{i} \leq s$.

 We then have the following analogue of Proposition \ref{prop 7}:

 \begin{proposition}\label{prop 12}
  Suppose that $X: [0,\infty] \to s\mathbf{Set}$ is a functor. Then the ep-metric space $\Re(X)_{s}$ has underlying set given by the set of vertices of $X_{s}$, with metric defined within path components by (\ref{eq 10}). Elements $x,y$ in distinct path components have $d(x,y) = \infty$.  
  \end{proposition}

 The map $X_{s} \to X_{\infty}$ defines a map $\Re(X)_{s} \to \Re(X)$ and $\Re(X)_{\infty} = \Re(X)$. There is an isomorphism of ep-metric spaces
 \begin{equation}
   \varinjlim_{s}\ \Re(X)_{s} \xrightarrow{\cong} \Re(X),
   \end{equation}
since the element $\infty$ is terminal in $[0,\infty]$ and $\Re(X)_{\infty} = \Re(X)$.

\begin{example}[Partial metrics for Vietoris-Rips complexes]\label{ex 13}
  Suppose that $X$ is a finite totally ordered ep-metric space. Consider the associated functor $V_{\ast}(X): [0,\infty] \to s\mathbf{Set}$.
  
  The associated ep-metric space $\Re(X)_{s}$ has underlying set $X$. We have $d(x,y) = \infty$ if $x,y$ are in distinct path components of $V_{s}(X)$. Otherwise
  \begin{equation*}
    d(x,y) =  \inf_{P}\ \{ \sum\ d(x_{i},x_{i+1}) \},
  \end{equation*}
  indexed over all polygonal paths
  \begin{equation*}
    P: x=x_{0},x_{1}, \dots ,x_{n}=y,
  \end{equation*}
  with $d(x_{i},x_{i+1}) \leq s$.
  
  If $d(x,y) = t \leq s$ in $X$ then $d(x,y)=t$ in $\Re(X)_{s}$. Otherwise, the distance between $x$ and $y$ in the same path component of $\Re(X)_{s}$ is more interesting --- it is achieved by a particular path $P$ since $X$ is finite, and $d(x,y)$ is a type of weighted path length.

  We see in Example \ref{ex 8} that there is an isomorphism of ep-metric spaces $\phi: \Re(V_{\ast}(X)) \xrightarrow{\cong} X$. It follows that there is an ep-metric space map $\phi_{s}: \Re(X)_{s} \to X$ which is the identity on the underlying point set $X$, and compresses distances.
  \end{example}

\section{The singular functor}

  The right adjoint $S$ of the realization functor $\Re$ is defined for an ep-metric space $Y$ by
  \begin{equation*}
    S(Y)_{s,n} = \hom(U_{s}^{n},Y),
  \end{equation*}
  where $\hom(U_{s}^{n},Y)$ is the collection of ep-metric space morphisms $U^{n}_{s} \to Y$.
  Equivalently, $S(Y)_{s,n}$ is the set of families of points $( x_{0}, x_{1}, \dots ,x_{n} )$ in $Y$ such that $d(x_{i},x_{j}) \leq s$.

A simplex $( x_{0}, x_{1}, \dots ,x_{n} )$ is alternatively a function $\mathbf{n} \to Y$ (a ``bag of words''), with a distance restriction. There is no requirement that the elements $x_{i}$ are distinct. This simplex is non-degenerate if and only if $x_{i} \ne x_{i+1}$ for $0 \leq i \leq n-1$.
\medskip

Suppose that an ep-metric space $X$ is totally ordered, as in Example \ref{ex 8} above.
Then, in view of the discussion of Example \ref{ex 8}, the canonical map $\eta: V_{\ast}(X) \to S\Re(V_{\ast}(X))$ consists of functions
$\eta: V_{t}(X) \to S_{t}(X)$ which send simplices $\sigma :x_{0} \leq x_{1} \leq \dots \leq x_{n}$ with $d(x_{i},x_{j}) \leq t$ to the list of points $(x_{0},x_{1}, \dots ,x_{n})$.

If $\sigma$ is non-degenerate, so that the vertices $x_{i}$ are distinct, then $\eta(\sigma)$ is a non-degenerate simplex of $S_{t}(X)$.
\medskip

The poset $NZ$ of non-degenerate simplices of a simplicial set $Z$ has $\sigma \leq \tau$ if there is a subcomplex inclusion $\langle \sigma \rangle \subset \langle \tau \rangle$, where $\langle \sigma \rangle$ is the subcomplex of $Z$ which is generated by the simplex $\sigma$. Equivalently, $\sigma \leq \tau$ if there is an ordinal number map $\theta$ such that $\theta^{\ast}(\tau) = \sigma$. 

The map $\eta$ induces a morphism $\eta_{\ast}: NV_{t}(X) \to NS_{t}(X)$ of posets of non-degenerate simplices.

\begin{lemma}\label{lem 14}
Suppose that $X$ is a totally ordered ep-metric space.  Then the induced simplicial set map
\begin{equation*}
  \eta_{\ast}: BNV_{t}(X) \to BNS_{t}(X)
\end{equation*}
of associated nerves is a weak equivalence.
\end{lemma}

\begin{proof}
Given a non-degenerate simplex $\sigma \in S_{t}(X)$, write $L(\sigma)$ for its list of distinct elements.

Suppose that $\langle \tau \rangle \subset \langle \sigma \rangle$, where $\tau$ and $\sigma$ are non-degenerate simplices of $S_{t}(X)$. Then $\tau = s \cdot d(\sigma)$ for an (iterated) face map $d$ and degeneracy $s$.
Then
\begin{equation*}
  L(\tau) = L(s \cdot d(\sigma)) = L(d(\sigma)) \subset L(\sigma).
\end{equation*}
It follows that the assignment $\sigma \mapsto L(\sigma)$ defines a poset morphism
\begin{equation*}
  L: NS_{t}(X) \to NV_{t}(X).
\end{equation*}
The composite
\begin{equation*}
  NV_{t}(X) \xrightarrow{\eta} NS_{t}(X) \xrightarrow{L} NV_{t}(X)
\end{equation*}
is the identity on $NV_{t}(X)$.

Consider the composite poset morphism
\begin{equation}\label{eq 12}
  NS_{t}(X) \xrightarrow{L} NV_{t}(X) \xrightarrow{\eta} NS_{t}(X).
\end{equation}
Given a non-degenerate simplex $\tau = (y_{0},\dots y_{r})$ of $S_{t}(X)$, write $L(\tau) = ( s_{0}, \dots s_{k})$ for the list of distinct elements of $\tau$, in the order specified by the total order for $X$. Then the list
\begin{equation*}
  V(\tau) = (y_{0}, \dots ,y_{r},s_{0}, \dots ,s_{k})
\end{equation*}
is a simplex of $S_{t}(X)$, since each $s_{j}$ is some $y_{i_{j}}$, and there are relations
\begin{equation*}
  \langle \tau \rangle \leq \langle V(\tau) \rangle \geq \langle L(\tau) \rangle
  \end{equation*}
as subcomplexes of $S_{t}(X)$.

The simplex $V(\tau)$ has the form $V(\tau) = s(V_{\ast}(\tau))$ for a unique iterated degeneracy $s$ and a unique non-degenerate simplex $V_{\ast}(\tau)$ (see Lemma \ref{lem 18}), and $\langle V(\tau) \rangle = \langle V_{\ast}(\tau) \rangle$.

 Suppose that $\gamma$ is non-degenerate in $S_{t}(X)$ and that $\gamma \in \langle
 \tau \rangle$. Then $\gamma = d(\tau)$ for some face map $d$, and $\gamma = (x_{0}, \dots ,x_{k})$ is a sublist of $\tau = ( y_{0}, \dots ,y_{r})$. The ordered list $L(\gamma)$ of distinct elements of $\gamma$ is a sublist of $L(\tau)$, and $V(\gamma)$ is a sublist of $V(\tau)$. There is a diagram of relations
 \begin{equation*}
   \xymatrix{
     \langle \tau \rangle \ar[r] & \langle V_{\ast}(\tau) \rangle
     & \langle L(\tau) \rangle \ar[l] \\
     \langle \gamma \rangle \ar[r] \ar[u]
     & \langle V_{\ast}(\gamma) \rangle \ar[u]
     & \langle L(\gamma) \rangle \ar[l] \ar[u]
   }
 \end{equation*}
 
 It follows that the composite (\ref{eq 12}) is homotopic to the identity on the poset $NS_{t}(X)$, and the Lemma follows.
\end{proof}

The subdivision $\sd(Z)$ of a simplicial set $Z$ is defined by
\begin{equation*}
  \sd(Z) = \varinjlim_{\Delta^{n} \to Z}\ BN\Delta^{n}.
\end{equation*}
The poset morphisms $N\Delta^{n} \to NZ$ that are induced by simplices $\Delta^{n} \to Z$ together induce a map
\begin{equation*}
  \pi: \sd(Z) \to BNZ.
\end{equation*}
It is known \cite{J34} (and not difficult to prove) that the map $\pi$ is a bijection for simplicial sets $Z$ that are polyhedral.

A polyhedral simplicial set is a subobject of the nerve of a poset. All oriented simplicial complexes are polyhedral in this sense. Examples include the Vietoris-Rips systems $V_{s}(X)$ associated to a totally ordered ep-metric space $X$, since
\begin{equation*}
  V_{s}(X) \subset V_{\infty}(X) = BX,
\end{equation*}
where $BX$ is the nerve of the totally ordered poset $X$.

\begin{lemma}\label{lem 15}
  Suppose that $X$ is an ep-metric space.
  Then the map
\begin{equation*}
  \pi: \sd(S_{t}(X)) \to BNS_{t}(X)
\end{equation*}
  is a weak equivalence.
  \end{lemma}

\begin{proof}
    We show that all subcomplexes $\langle \sigma \rangle$ which are generated by non-degenerate simplices $\sigma$ of $S_{t}(X)$ are contractible. Then
  Lemma 4.2 of \cite{J34} implies that the map $\pi$ is a weak equivalence.

  A non-degenerate simplex $\sigma$ has the form $\sigma = (x_{0},x_{1}, \dots ,x_{k})$ with $x_{i} \ne x_{i+1}$. The simplices $\tau$ of $\langle \sigma \rangle$ have the form
  \begin{equation*}
    \tau = \theta^{\ast}\sigma = (x_{\theta(0)}, \dots ,x_{\theta(k)}),
  \end{equation*}
  where $\theta: \mathbf{k} \to \mathbf{n}$ is an ordinal number morphism.

  For each such $\theta$, the list
  \begin{equation*}
    (x_{0},x_{\theta(0)}, \dots, x_{\theta(k)})
  \end{equation*}
  defines a simplex $\tau_{\ast}$ of $\langle \sigma \rangle$, since
  $\tau_{\ast}$ is a face of the simplex
  \begin{equation*}
    s_{0}(\sigma) = (x_{0},x_{0}, \dots ,x_{k}).
  \end{equation*}

  The simplices $\tau_{\ast}$ define functors
\begin{equation*}
  \xymatrix{
    x_{0} \ar[r] \ar[d] & x_{0} \ar[r] \ar[d] & \dots \ar[r] & x_{0} \ar[d] \\
    x_{\theta(0)} \ar[r] & x_{\theta(1)} \ar[r] & \dots \ar[r] & x_{\theta(m)}
  }
  \end{equation*}
or homotopies, that consist of simplices of $\langle \sigma \rangle$ that patch together to give a contracting homotopy $\langle \sigma \rangle \times \Delta^{1} \to \langle \sigma \rangle$.
  \end{proof}

\begin{theorem}\label{th 16}
  Suppose that $X$ is a totally ordered ep-metric space. Then there is a diagram of weak equivalences
  \begin{equation*}
    \xymatrix{
      BNV_{t}(X) \ar[d]_{\eta_{\ast}} & \sd(V_{t}(X)) \ar[l]_{\pi}^{\cong} \ar[r]^-{\gamma} \ar[d]^{\eta_{\ast}} & V_{t}(X) \ar[d]^{\eta} \\
      BNS_{t}(X) & \sd(S_{t}(X)) \ar[l]^{\pi} \ar[r]_-{\gamma} & S_{t}(X)
    }
    \end{equation*}
    In particular, the map $\eta: V_{t}(X) \to S_{t}(X)$ is a weak equivalence.

    This diagram is natural in $t$.
\end{theorem}

\begin{proof}
   The map $\eta_{\ast}: BNV_{t}(X) \to BNS_{t}(X)$ is a weak equivalence by Lemma \ref{lem 14}. The map $\pi: \sd(S_{t}(X)) \to BNS_{t}(X)$ is a weak equivalence by Lemma \ref{lem 15}.
   The instances of the maps $\gamma$ are weak equivalences \cite{J34}. It follows that the maps
   $\eta_{\ast}: \sd(V_{t}(X)) \to \sd(S_{t}(X))$ and $\eta: V_{t}(X) \to S_{t}(X)$ are weak equivalences.
\end{proof}  

\begin{remark}
  The total ordering on the ep-metric space $X$ in Theorem \ref{th 16} is intimately involved in the definition of the Vietoris-Rips system $V_{\ast}(X)$, the morphism
\begin{equation*}
  \eta: V_{t}(X) \to S_{t}(\Re(V_{\ast}(X))) = S_{t}(X),
\end{equation*}
and all induced maps $\eta_{\ast}$.

The counit $\eta: Z \to S(\Re(Z))$ is not a sectionwise weak equivalence in general. One can show that $S_{\infty}(\Re(Z))$ is the nerve of the trivial groupoid on the vertex set of $Z_{\infty}$ (Proposition \ref{prop 7}), and is therefore contractible, whereas the space $Z_{\infty}$ may not be contractible. For example, if $K$ is a simplicial set, then  there is an identification $K = (L_{s}K)_{\infty}$.

It may be that the map
\begin{equation*}
  \eta: BP_{t}(X) \to S_{t}(\Re(BP_{\ast}(X)))
\end{equation*}
is a weak equivalence for arbitrary ep-metric spaces $X$, but this has not been proved. Such a result would give a non-oriented version of Theorem \ref{th 16}.
  \end{remark}

The following is a classical result, which is included here for the sake of completeness. This result is usually neither expressed nor proved in the form displayed here.

\begin{lemma}\label{lem 18}
Suppose that $\sigma$ is an $n$-simplex of a simplicial set $X$. Then there is a unique iterated degeneracy and a non-degenerate simplex $x$ such that $\sigma = s(x)$.
\end{lemma}

An iterated degeneracy is a surjective ordinal number map $s: \mathbf{n} \to \mathbf{k}$.  Such a map induces a function $s: X_{k} \to X_{n}$ for a simplicial set $X$. Lemma \ref{lem 18} says that $\sigma = s(x)$ for some iterated degeneracy $s$ and a non-degenerate simplex $x$, and that this representation is unique.

\begin{proof}[Proof of Lemma \ref{lem 18}]
Suppose that $\sigma = s(x) = s'(x')$ where $s,s'$ are iterated degeneracies and $x,x'$ are non-degenerate.

    The $x = d(\sigma)$ for some face map $d$ such that $d \cdot s = 1$, and so $d(s'(x')) = s''(d''(x))$ for some iterated degeneracy $s''$ and face map $d''$. But $x$ is non-degenerate, so that $s''=1$ and $x = d''(x')$. Similarly, $x' = \tilde{d}(x)$ for some face map $\tilde{d}$. But then $x$ and $x'$ have the same dimension, and $d=1$, so that $x=x'$.

    If $s \ne s'$ there is a face map $d$ such that $d \cdot s=1$ but $d \cdot s' \ne 1$. Then $\sigma =s(x) = s'(x)$ for $s \ne s'$ and $x$ non-degenerate, then
    \begin{equation*}
      d(\sigma) = x = d(s'(x)) = s''(d''(x))
      \end{equation*}
    for some degeneracy $s''$ and face map $d''$, at least one of which is non-trivial.

    But $x$ is non-degenerate, so that $s''=1$ and $x=d''(x)$ only if $d'' =1$. This contradicts the assumption that $s \ne s'$. 
\end{proof}



\bibliographystyle{plain}
\bibliography{spt}

\end{document}